\documentclass[10pt, leqno]{article}
\usepackage{amsmath,amssymb,amsthm,hyperref,bbm,enumitem,soul}
\usepackage[final,expansion=true,protrusion=true,spacing=true,kerning=true]{microtype}
\microtypecontext{spacing=nonfrench}
\usepackage{graphicx,bbold}

\def\spine{1.1in}
\usepackage[heightrounded, twoside, top=1in, bottom=1.3in, inner=\spine, outer=\spine]{geometry} 

\def\cl{\text{\rm cl}}
\def\V{\mathcal V}
\def\supp{\text{\rm supp}\,}

\def\epsilon{\varepsilon}
\def\restr#1{{\lower3.5pt\hbox{$\big|_{#1} $}}}

\newtheorem{theorem}{Theorem}
\newtheorem*{theorem*}{Theorem}
\newtheorem{lemma}[theorem]{Lemma}

\newtheorem{proposition}[theorem]{Proposition}
\newtheorem*{proposition*}{Proposition}
\newtheorem{corollary}[theorem]{Corollary}

\title{Discreteness of the minimizers of weakly repulsive\\ interaction energies on Riemannian manifolds}
\author{Oleksandr Vlasiuk}
\begin{document}
\maketitle
\begin{abstract}
    It is shown that the supports of measures minimizing weakly repulsive energies on Riemannian manifolds with sectional curvature bounded below do not have concentration points. This extends the results of Bj\"orck and Carrillo, Figalli, and Patacchini for such energies on the Euclidean space, and complements the results about the discreteness of minimizers of the geodesic Riesz energy on the sphere by Bilyk, Dai, and Matzke.
\end{abstract}
\bigskip

\section{Introduction} 
We are interested in characterizing measures on which the minimal value of the functional
\begin{equation}
    \label{eq:energy}
    I_F(\mu) = \iint_\Omega F(\rho(x,y)) \, d\mu(x) d\mu(y) \qquad \text{ s.t. } \mu\in\mathcal P(\Omega)
\end{equation}
is attained; here $ \Omega $ stands for a metric measure space with distance $ \rho $, and $ I_F $ is defined on the set of Borel probability measures on $ \Omega $, denoted by $ \mathcal P(\Omega) $. Expression \eqref{eq:energy} is the self-interaction energy of the measure $ \mu $. It  can be used to model a large collection of pairwise interacting agents, in which the features of individual agents are irrelevant: asteroid clouds, flocks of birds, particles in a suspension, etc.   
Carrillo, Figalli, and Patacchini \cite{carrilloGeometry2017} have established that when $ \Omega = \mathbb R^d $ with the Euclidean distance and the kernel $ F $ satisfies
\[
    F'(t) t^{1-\alpha} \to -C, \qquad t\to 0+, \quad \alpha > 2,
\]
with $ F(t) \geq 0, t\geq R $ for a fixed positive $ R $, then the local minimizers of $ I_F $ in $ d_\infty $ topology are discrete. In the present note we relax the assumptions on $ F(t) $ compared to \cite{carrilloGeometry2017}, and show that the local minimizers do not have concentration points also when $ \Omega $ is a more general Riemannian manifold. Namely, we show that if $ F(t) = -(1+o(1))\,Ct^{\alpha} $ and $ \Omega $ is a Riemannian manifold with sectional curvature bounded below, then the supports of local minimizers of $ I_F $ do not have concentration points on $ \Omega $.

For a function $ F:[0,+\infty)\to [0,+\infty] $ and a metric measure space $ \Omega $ with distance function $ \rho(\cdot\,, \cdot\,) $, we define the of value of the {\it interaction energy} $ I_F(\mu) $ on  a Borel probability measure $ \mu\in\mathcal P(\Omega) $ by \eqref{eq:energy} with $ \rho(x,y) $ denoting the distance between $ x,y\in\Omega $.  
If $ F $ is decreasing in a neighborhood of zero, we will call the energy $ I_F(\mu) $ {\it repulsive}, to reflect that the optimization of $ I_F(\mu) $ results in measures for which the set of pairwise distances,
\[ 
    \{ \rho(x,y): x,y\in\supp \mu \},
\]
favors larger values.
In other words, the minimizing measure tends to have supports with distance sets avoiding large values of $ F$. The stronger the repulsion at the origin, the more minimizing measures will avoid forming clusters. Let us illustrate this trend in the case $ \Omega = \mathbb S^d $, where Bilyk, Dai, and Matzke \cite{bilykStolarsky2016,bilykGeodesic2016} obtain the following characterization of the measures that minimize $ I_{F_\delta} $ with $ F_\delta(t) = - \text{sgn\,}\delta\,\cdot t^\delta $.

\begin{figure}[t]
    \centering
    \includegraphics[width=0.3\linewidth]{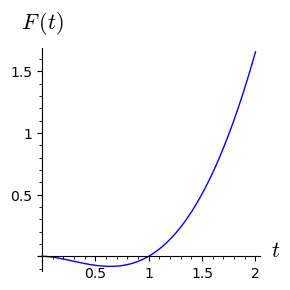}
    \hfil
    \includegraphics[width=0.3\linewidth]{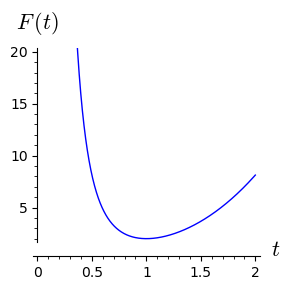}
    \caption{Examples of profiles of radial attractive-repulsive kernels. Left: weakly repulsive; right: strongly repulsive kernel.}%
    \label{fig:profiles}
\end{figure}

\begin{theorem}[\cite{bilykGeodesic2016,bilykStolarsky2016}]
    \label{thm:bilyk}
    Let $ \Omega = \mathbb S^d $ with $ \rho $ the geodesic distance, then for the functional
    \[
        I_{F_\delta}(\mu) = \iint_{\mathbb S^d} -{\rm sgn\,}\delta\, \cdot (\rho(x,y))^\delta \, d\mu(x) d\mu(y) \qquad \text{ \rm s.t. }  \mu \in\mathcal P(\mathbb S^d)
    \]
    there holds:\\
    (i) $-d < \delta < 1,\ \delta \neq 0$: the unique minimizer is the uniform measure on $ \mathbb S^d $; \\
    (ii) $\delta = 1$: any centrally symmetric measure is a minimizer; \\
    (iii) $\delta>1$: any measure of the form $ \mu = (\delta_x + \delta_{-x})/2 $ with $ x\in \mathbb S^d $ is a minimizer.
\end{theorem}
\noindent In the above statement we take into account that for positive exponents $ \delta $, the problem of minimizing the energy with the kernel $ F(t) = t^\delta $ is trivial, and consider instead the maximizers of the corresponding energy.  Observe that the kernels with sufficiently small $ \delta $ result in strongly repulsive energies, which are minimized on the uniform measure; in particular, the strong repulsion prevents clustering in the minimizer. In this note, we are concerned with the weakly repulsive energies, and the phenomenon of clustering for their minimizers, which corresponds to the case (iii) in the above proposition.

When $ \Omega = \mathbb R^d $, for the minimizing measures to be well-defined, it is necessary to consider kernels which confine the minimizers to a compact set. This can be achieved by assuming that $ F(t)\to\infty, t\to \infty $. Typical profiles of such {\it attractive-repulsive} kernels are shown in Figure~\ref{fig:profiles}.  
In the present note the only global assumptions on $ F $ is continuity, and that the particular local minimizers considered below are well-defined. The main local assumption is that
\begin{equation}
    \label{eq:weakly_rep}
    F(t) = -(C+o(1))\,t^{\alpha}, \qquad t \to 0+, \quad \alpha>2, \ C >0,
\end{equation}
that is, $ \lim_{t\to 0+} F(t)t^{-\alpha} = -C$ for some $ C>0 $ and $ \alpha > 2 $. We refer to the energies $ I_F $ with $ F $ satisfying the above equality as {\it weakly repulsive}. The behavior of such energies should be contrasted with what we will call here {\it strongly repulsive}, or singular, energies, for which \eqref{eq:weakly_rep} takes the form
\[
    F(t) = (C+o(1))\,t^{\alpha}, \qquad t \to 0+, \quad \alpha<0, \ C >0.
\]
Note that for strongly repulsive kernels, we need to additionally assume integrability on $ \Omega $.
As seen in point {\it (i)}\/ of the above theorem, the minimizers of strongly repulsive energies are generally not discrete.

Our main Theorem~\ref{thm:main} below implies the following statement. Unlike Theorem~\ref{thm:bilyk}, the kernel does not have to be given by a power law; the resulting characterization of the minimizers is however less precise.
\begin{proposition}
    Let the continuous kernel $F:[0,\infty) \to [0, \infty) $ be decreasing in a neighborhood of zero and satisfy \eqref{eq:weakly_rep}, and let $ \rho $ be the geodesic distance on $ \mathbb S^d $. If $ \widetilde \mu $ is a local minimizer of the functional
    \[
        I_F(\mu) = \iint_{\mathbb S^d} F(\rho(x,y)) \, d\mu(x) d\mu(y) \qquad \text{ \rm s.t. }  \mu \in\mathcal P(\mathbb S^d)
    \]
    in the $ d_\infty $ topology, then $ \supp \widetilde \mu $ is finite.
\end{proposition}

\section{Main result}

Recall the definition of Wasserstein $ d_\infty $ distance between measures $ \mu , \nu \in \mathcal P(\Omega) $:
\[
    d_\infty(\mu,\nu) = \inf \left\{ \sup_{(x,y)\in\supp \pi} \rho(x,y) : \pi \in \Pi(\mu,\nu) \right\},
\]
where $ \Pi(\mu,\nu) $ is the set of probability measures on $ \Omega \times\Omega $, with the first and second marginals equal to $ \mu $ and $ \nu $, respectively. Observe that convergence in $ d_\infty $ distance implies convergence in any $ d_p $ distance, $ 1\leq p <\infty $, and in the weak$ ^* $ topology. The converse does not hold; in particular, the function $ \gamma:[0,1] \to \mathcal P(\Omega) $, given by
\begin{equation}
    \label{eq:curve}
    \gamma(t) = (1-t)\mu_1 + t\mu_2,
\end{equation}
is not a continuous curve in $ \mathcal P(\Omega) $ with the $ d_\infty $ distance, whereas it is in $ \mathcal P(\Omega) $ equipped with the weak$^*$ topology. Furthermore, since $ d_\infty $ induces the strongest topology of those mentioned above, a measure minimizing $ I_F $ locally in  the weak$ ^* $ or $ d_p $ topologies for finite $ p $ is also minimizing it locally in the $ d_\infty $ topology.

We will need the following almost standard fact \cite{Bjorck1956,carrilloGeometry2017}; it can be viewed as a variational analog of the positivity of the second derivative at a point of local minimum. 
In the lemma, we refer to a set of signed Borel measures $ \V_0 $ as {\it 0-symmetric}, if the inclusion  $ \nu \in \V_0 $ implies $ -\nu\in\V_0 $. We also write 
\[
    \mathbb R \V_0 := \{ \lambda \nu : \lambda \in \mathbb R,\ \nu \in \mathcal V_0 \},
\] 
and denote by $ \cl^* (\mathbb R \V_0) $ its closure in the weak$ ^* $ topology.  
\begin{lemma}
    \label{lem:quadratic}
    Let $ F:[0,\infty) \to [0, \infty) $ be continuous and $ \Omega $ be a metric measure space. Assume that $ \widetilde \mu $ is a minimizer of $ I_F $ on the set $ \widetilde \mu + \V_0 \subset\mathcal P(\Omega) $, where $ \V_0 $ is a 0-symmetric set of signed Borel measures $ \nu $ with $ \nu(\Omega) =0 $.
    Then for any signed $ \nu \in \cl^*(\mathbb R \V_0) $ there holds
    \begin{equation}
        \label{eq:2nd_order}
        \iint_\Omega F(\rho(x,y)) \, d\nu(x) d\nu(y) \geq 0.
    \end{equation}
\end{lemma}
\begin{proof}
    Pick a $ \nu \in \V_0$. Since $ \widetilde\mu $ is a minimizer on  $ \widetilde \mu + \V_0 $, $ I_F(\widetilde\mu) \leq I_F(\widetilde\mu\pm\nu)$, that is,
    \[
        \iint_\Omega F(\rho(x,y)) \, d\widetilde \mu(x) d\widetilde \mu(y) \leq \iint_\Omega F(\rho(x,y)) \, d(\widetilde \mu\pm\nu)(x)\,d(\widetilde \mu\pm\nu)(y). 
    \]
    Distributing the integral gives
    \[
        \pm2\iint_\Omega F(\rho(x,y)) \, d\widetilde \mu(x) d\nu(y)  + \iint_\Omega F(\rho(x,y)) \, d\nu(x) d\nu(y) \geq 0.
    \]
    Note that the first term in the above equation is linear in $ \nu $, the second term is quadratic. The inequality must hold for both measures $ \widetilde \mu \pm \nu $, so we conclude that $ \iint F(\rho(x,y)) \, d\nu(x) d\nu(y) \geq 0 $, since otherwise changing $ \nu $ to $ -\nu $ would make the left-hand side negative. This gives \eqref{eq:2nd_order} for $ \nu \in \V_0 $.

    Note further that if \eqref{eq:2nd_order} holds for a measure $ \nu $, it also holds for $ \lambda \nu $, with any $ \lambda \in \mathbb R $; this gives \eqref{eq:2nd_order} for $ \nu \in \mathbb R\V_0 $. Finally, the continuity of $ F $ implies that the functional $ I_F $ is continuous in the weak$ ^* $ topology, from which \eqref{eq:2nd_order} follows for the weak$ ^* $ closure of $ \mathbb R\V_0 $.  
\end{proof}
\begin{corollary}
    \label{cor:apply_lem}
    If $ \widetilde \mu $ is a local  minimizer of $ I_F $ in a $ d_\infty $-ball in $ \mathcal P(\Omega) $ of radius $ \epsilon $ around $ \widetilde \mu $, then for any signed $ \nu $ with $ \nu(\Omega) = 0 $ and $ \supp \nu \subset \supp \widetilde\mu $, satisfying $ \supp \nu \subset B(x_0,\epsilon) $ for some $ x_0\in\Omega $, inequality \eqref{eq:2nd_order} holds.
\end{corollary}
\begin{proof}
    Fix the $ x_0\in \Omega $ and the $ \rho $-ball $ B(x_0,\epsilon) $ around it; denote this ball by $ B $. Note that $ B \cap \supp \widetilde\mu \neq \emptyset $.
    Let further $ \V_0 $ be a $ 0 $-symmetric set of signed Borel measures satisfying $ \nu_0(\Omega) = 0 $ and $ \supp \nu_0 \subset B $, such that $ \widetilde\mu + \V_0 \subset \mathcal P(\Omega) $. Note that $ \V_0 $ is nonempty, since it contains measures of the form 
    \[ 
        \nu_0 = \sum_i c_i \widetilde\mu\restr{B_i}\qquad  -1 \leq c_i \leq 1,\ \nu_0(B)=0,
    \]
    for closed disjoint balls $ B_i \subset B $ that have nonempty intersection with $ \supp \widetilde\mu $.
    In addition, for a $ \nu_0 \subset \V_0 $ there holds $ d_\infty(\widetilde\mu, \widetilde\mu+\nu_0) \leq \epsilon $. It follows that $ \widetilde \mu $ is a minimizer of $ I_F $ on the set $ \widetilde \mu + \V_0 $.

    To complete the proof using Lemma~\ref{lem:quadratic}, we have to show that any signed $ \nu $ as in the statement of the corollary is contained in $ \cl^* (\mathbb R \V_0) $.
    First observe that any signed measure $ \nu $ on $ \Omega $ can be approximated by discrete finitely-supported measures on $ \Omega $ in the weak$^*  $ topology. It therefore suffices to approximate a discrete measure $ \nu' $ supported inside $ B\cap \supp \widetilde\mu  $ in the weak$ ^* $ topology with a sequence of measures $ \nu_n \in \mathbb R \V_0 $.  

    We conclude by noting that such approximation of a finitely-supported $ \nu' = \sum_i \alpha_i\delta_{x_i} $ is given by
    \[
        \nu_n = \sum_{i} \frac{\alpha_i }{ \widetilde\mu(B_i) }{\widetilde\mu\restr{B_i}}, \qquad n \geq 1,
    \] 
    where $ B_i \subset B $ are closed disjoint $ \rho $-balls centered at $ x_i $, of radius at most $ 1/n $.

\end{proof}
\noindent Observe that the proof of Corollary~\ref{cor:apply_lem} applies also when the assumptions on $ \widetilde \mu $ in Lemma~\ref{lem:quadratic} are strengthened to being a local weak$ ^* $ minimizer. In this case, the assumption of $ \supp \nu \subset B(x_0,\epsilon) $ can be dropped. This results in the following
\begin{corollary}
    \label{cor:weakstar}
    If $ \widetilde \mu $ is a local weak$ ^* $ minimizer of $ I_F $ in $ \mathcal P(\Omega) $, then for any signed $ \nu $ with $ \nu(\Omega) = 0 $ and $ \supp \nu \subset \supp \widetilde\mu $, inequality \eqref{eq:2nd_order} holds.
\end{corollary}
In addition, when $ \widetilde \mu $ is a local minimizer in the weak$ ^* $ topology, its potential $ \int_\Omega F(\rho(x,y))\, d\widetilde \mu $ is constant on $ \supp \widetilde \mu $, which gives
\begin{equation}
    \label{eq:independence} 
    \iint_\Omega F(\rho(x,y)) \, d\widetilde \mu d\nu = 0, \qquad \supp \nu \subset \supp\widetilde\mu, \ \nu(\Omega) =0.
\end{equation}

We now apply equation~\eqref{eq:independence} to prove an interesting fact about weak$ ^* $ minimizers, which will not be used to obtain the main result below (see also \cite[Proposition 7.2]{bilykEnergy2019}).
\begin{proposition}
    If $ \mu_1, \mu_2 $ are local minimizers of $ I_F $ in the weak$ ^* $ topology and $ I_F(\mu_1) \neq I_F(\mu_2)  $, then it cannot hold that $ \supp \mu_1 \subset \supp\mu_2  $.
\end{proposition}
\begin{proof}
    Assume by contradiction that $ \supp \mu_1 \subset \supp\mu_2  $. By Corollary~\ref{cor:weakstar}, $ I_F(\mu_1-\mu_2) \geq 0 $. Then
    \[ 
        I_F (\mu_1) = I_F\left(\mu_2 +  (\mu_1-\mu_2)\right) = I_F(\mu_2) + I_F(\mu_1-\mu_2)  \geq I_F(\mu_2),
    \]
    where we used \eqref{eq:independence} for the cross terms in the second equality. This gives $ I_F(\mu_1)\geq I_F(\mu_2) $. 

    Let us now show that $ I_F(\mu_1)\leq I_F(\mu_2) $, which will result in a contradiction. Consider the map $ \gamma(t) $ given by equation \eqref{eq:curve}. Since $ \mu_1 $ is a local minimizer and $ \gamma(t) $ is a continuous curve, there exists a $ t_0\in (0,1) $ such that $ I_F(\mu_1) \leq I_F(\gamma(t)) $ whenever $ 0\leq t\leq t_0 $. Since $ \int_\Omega F(\rho(x,y))\, d\mu_2 = I_F(\mu_2) $ on $ \supp \mu_2 $, distributing the terms in $I_F(\gamma(t))$ gives
    \[ 
        I_F (\mu_1) \leq I_F(\gamma(t)) = (1-t)^2I_F(\mu_1) + 2 t(1-t)I_F(\mu_2) + t^2I_F(\mu_2), \qquad 0\leq t \leq t_0,
    \]
    which after rearrangement becomes $ (2t-t^2)I_F(\mu_1)\leq (2t-t^2)I_F(\mu_2) $. This yields the desired inequality and completes the proof.
\end{proof}

The reason to introduce the smooth structure on $ \Omega $ in the following theorem is that geodesics and angles between them have to be defined. This restriction can be lifted by using Alexandrov angles, thereby extending the result of the theorem to general complete geodesic metric spaces with curvature bounded below, see Section~\ref{sec:discussion} below.

\begin{theorem}
    \label{thm:main}
    Let $ \Omega $ be a smooth Riemannian manifold. Assume that the continuous kernel $ F:[0,\infty)\to[0,\infty) $ is decreasing in a neighborhood of zero and satisfies
    \[
        F(t) = -(C+o(1))\,t^{\alpha}, \qquad t \to 0+
    \]
    for $ C>0 $ and $ \alpha > 2 $, and that $ \widetilde \mu $ is a local minimizer of $ I_F $ in the $ d_\infty $ topology. If a point $ y\in\Omega $ is such that $ \Omega $ has sectional curvature bounded below in a neighborhood $ U_y \ni y $, then $ y $ is not a limit point of $ \supp \widetilde \mu $.
\end{theorem}

\begin{proof}
Assume  by contradiction that the support of the minimizer $ \widetilde \mu $ contains a sequence $ \{ y_k \} \subset \Omega $, converging to $ y\in\Omega $, and $ \Omega $ has curvature greater than $ -1/K^2 $ in $ U_y $. Without loss of generality, $ U_y $ and all the $ y_k $ lie within the injectivity radius of the exponential map $ \exp_y $ \cite{jostRiemannian2011} from $ y $, so we can consider $ \exp_y^{-1}: U_y \to \mathbb R^d $, the inverse of the exponential map.

By passing to a subsequence if necessary, we will make the following assumptions: (i) the directions of vectors $ \exp_y^{-1} y_k $ converge to the direction $ n_y \in \mathbb R^d $:
\begin{equation}
    \label{eq:directions}
    \frac{\exp_y^{-1} y_k}{\|\exp_y^{-1} y_k \|} \longrightarrow n_y, \qquad k\to\infty; 
\end{equation}
(ii) the distances $ \rho(y, y_k) $  are strictly decreasing. Here $ \|\cdot\| $ denotes the Euclidean norm in $ \mathbb R^d $. 

Let $ \Omega' $ be the hyperbolic plane of dimension $ \dim \Omega $ with constant sectional curvature $ -1/K^2 $, strictly smaller than that of $ \Omega $ in the neighborhood $ U_y $. Fix a point $ y'\in\Omega' $ and consider the map $ \psi = \exp_{y'} \circ \exp_y^{-1} $.
Both $ \exp_y^{-1} $ and $ \exp_{y'} $ are distance preserving on radial geodesics from $ y $, so $ \rho(y,x) = d'(y', x') $ with $ x\in U_y $, where the prime denotes the image of the corresponding point under $ \psi $. Similarly, the angles between geodesics passing through $ y $ are preserved under $ \psi $.
Consider a triple $ y $, $ y_k $, $ y_l $, $ l>k $. Let 
\[ 
    \begin{aligned}
        \rho(y_k,y_l) &= s_{kl}\\
        \rho(y_l, y)   &= s_l\\
        \rho(y,y_k)   &= s_k.
    \end{aligned}
\]
Because within the injectivity neighborhood of $ \exp_y $ geodesics are unique, there holds $ s_k \leq s_{kl} + s_l $, with equality if and only if
\[
    y_l \in \gamma(y, y_k),
\]
where $ \gamma(\cdot\,,\cdot) $ denotes the geodesic connecting the two points. 

Let us show that for $ k,l $ large enough, $ s_{kl} < s_k  $. Since by assumption (ii) $ s_l < s_k $,  this will imply that the geodesic $ \gamma(y,y_k) $ of length $ s_k $ is the longest side in the triangle $ \triangle y_kyy_l $.
Indeed, using the law of cosines in $\Omega'$ \cite[I.2.13]{bridsonMetric1999}, we have 
\addtocounter{equation}{1}
\begin{align*}
    \cosh \frac{d'(y_k', y_l')}{K} 
        &= \cosh \frac{d'(y', y_k')}{K} \cosh \frac{(y', y_l')}{K} - \sinh \frac{(y', y_k')}{K} \sinh \frac{(y', y_l')}{K} \cos(\angle y_k'y'y_l' ) \\
        &= \cosh \frac{s_k}{K} \cosh \frac{s_l}{K} - \sinh \frac{s_k}{K} \sinh \frac{s_l}{K} \cos(\angle y_kyy_l ) \\
        &= C_k C_l - \sqrt{C_k^2-1} \sqrt{C_l^2-1} \cos(\angle y_kyy_l )\\
        &= C_k (1+ (C_l-1)) - \sqrt{C_l-1} \cdot \sqrt{C_k^2-1} \sqrt{C_l+1} \cos(\angle y_kyy_l ), \tag{\the\value{equation}}\label{eq:long_side}
\end{align*}
where we write $ C_w = \cosh \frac{s_w}{K} $, $ w\in \{ k,l \} $ for brevity.
In view of \eqref{eq:directions}, 
\[
    \lim_{k\to\infty} \cos(\angle y_kyy_l ) = 1, \qquad l\geq k,
\]
it follows that the expression in \eqref{eq:long_side} is less than $ C_k $ for a fixed $ k $ and $ l $ sufficiently large, because for such $ l $, 
\[
    \sqrt{C_l-1} < \frac{1}{C_k} \cdot \sqrt{C_k^2-1} \sqrt{C_l+1} \cos(\angle y_kyy_l ).
\]
Since $ \cosh t $ is an increasing function for $ t\geq 0 $, it follows that $ d'(y_k', y_l') < s_k $. By construction, $ \psi(U_y) \subset \Omega' $ has lower sectional curvature than $  U_y \subset \Omega $, and so by the Toponogov's theorem \cite[p. 35]{cheegerComparison2008} there holds 
\[
    d'(y_k', y_l') \geq \rho(y_k, y_l) = s_{kl}, 
\]
which shows that $ s_{kl} < s_k $.

Having shown that $ \gamma(y, y_k) $ is the longest side in the geodesic triangle $ \triangle y_kyy_l \subset \Omega $, we use the construction from \cite{carrilloGeometry2017}, which implies that for any sufficiently small triangle in the support of a minimizer measure, the angle adjacent to the longest side cannot be too small. This will give a contradiction to \eqref{eq:directions}.

To establish a lower bound on the angle $ \angle y_kyy_l $, we argue as follows. By Corollary~\ref{cor:apply_lem}, Lemma~\ref{lem:quadratic} can be applied to signed measures supported on $ \supp \widetilde \mu $, if their supports can be contained in a sufficiently small ball. Thus, for $ s_k,s_{kl},s_l $ sufficiently small, equation \eqref{eq:2nd_order} can be applied with $ \nu = t\delta_{y_k} + (1-t)\delta_{y} - \delta_{y_l} $, giving
\[
    t(1-t) F(s_k) \geq tF(s_{kl}) + (1-t)F(s_l).
\]
Dividing both sides by $ F(s_k) $ yields
\[
    t(1-t) \leq ta + (1-t)b, \qquad t\in\mathbb R,
\]
for $ a = F(s_{kl})/F(s_k) $, $ b = F(s_l)/F(s_k) $. Optimizing over $ t \in \mathbb R $ shows that the above inequality is the strongest when $ t = - \frac{a-b-1}2 $, at the vertex of the parabola. Substituting this value for $ t $ produces $ b^2-2b(a+1)+(a-1)^2\leq 0 $. Since $ b $ has to lie to the right of the smaller root of the left-hand side, $ b\geq a-2\sqrt a +1 = (\sqrt a - 1)^2 $. Because $ F $ is decreasing in a neighborhood of zero and $ s_{kl} < s_k $, assuming again that $ s_k $ is sufficiently small gives $ a = F(s_{kl})/F(s_k) < 1 $. This implies
\[
    \sqrt a + \sqrt b \geq 1.
\]
Hence,
\begin{equation}
    \label{eq:region}
    \sqrt{-F(s_{kl})} + \sqrt{-F(s_l)} \geq \sqrt{-F(s_k)}.  
\end{equation} 
Taking into account the assumption \eqref{eq:weakly_rep} on the kernel $ F $, inequality \eqref{eq:region} can be rewritten as
\begin{equation}
    \label{eq:approx_region}
    (1+o(1))Cs_{kl}^{\alpha/2} + (1+o(1))Cs_l^{\alpha/2} \geq (1+o(1))Cs_k^{\alpha/2}, \qquad (s_k,s_{kl},s_l) \to 0.
\end{equation}
Consider the function
\[
    R(s_k, s_{kl}, s_l) := s_{kl}^{\alpha/2} + s_l^{\alpha/2} -s_k^{\alpha/2}.
\]
Because $ \alpha > 2 $, the equality $ s_k = s_{kl}+s_l $ implies
\[ 
    R(s_k, s_{kl}, s_l) < 0.
\]
By the above discussion about injectivity of the exponential map, this means that whenever $ y_l \in \gamma(y, y_k) $, there holds $ R(s_k, s_{kl}, s_l) < 0 $. Observe that $ R $ is homogeneous in its variables and rewrite it as
\[
    \begin{aligned}
        R(s_k, s_{kl}, s_l) &= s_k^{\alpha/2}\left(\left(\frac{s_{kl}}{s_k}\right)^{\alpha/2} + \left(\frac{s_l}{s_k}\right)^{\alpha/2} -1\right)\\
                       &=:s_k^{\alpha/2}\left(\beta_1^{\alpha/2} + \beta_2^{\alpha/2} -1\right).
\end{aligned}
\] 
By continuity of the above expression in $ \beta_i $, $ i=1,2 $, there exists an open neighborhood of $ \{ (\beta_1, \beta_2) : \beta_1+\beta_2 = 1 \} \subset \mathbb R^2 $, for which
\[ 
    R(s_k, s_k\beta_1, s_k\beta_2) < - P s_k^{\alpha/2}
\]
for a positive $ P $.
Denote this neighborhood by $ U $.  In view of how the ratios $ \beta_i $ are defined, there holds $ \beta_1+\beta_2 \geq 1 $ by the triangle inequality. Let $ c_U>0 $ be such that $ U\supset \{ (\beta_1, \beta_2) : 1\leq \beta_1+\beta_2 \leq  1+ c_U \} $.

Since $ s_k $ is the length of the longest side in the triangle $ \triangle y_kyy_l \subset \Omega $ when the angle $ \angle y_kyy_l $ is sufficiently small, for a fixed $ k $ and any sufficiently large $ l $ there holds $ s_l < c_U s_k $. Hence, for such $ k, l $ there holds
\[
    s_{kl} + s_l < (1+c_U)s_k,
\] 
and so, by the definition of $ c_U $, $R(s_k, s_k\beta_1, s_k\beta_2) < - P s_k^{\alpha/2}$. In view of \eqref{eq:approx_region}, the triple $ y, y_k,y_l $ then cannot belong to $ \supp \widetilde \mu $, because it would result in
\[
    - P s_k^{\alpha/2} + o(1)s_{kl}^{\alpha/2} + o(1)s_l^{\alpha/2} \geq o(1)s_k^{\alpha/2}, \qquad (s_k,s_{kl},s_l) \to 0,
\]
or, since $ s_k = \max\{ s_k, s_l, s_{kl} \} $,   
\[
    - P s_k^{\alpha/2} + o(1)s_{k}^{\alpha/2} \geq 0, \qquad s_k \to 0,
\]
a contradiction, since $ y, y_k, y_l $ were assumed to belong to $ \supp \widetilde \mu $. This finishes the proof of the theorem.
\end{proof}

\begin{corollary}
    For a closed compact smooth manifold $ \Omega $, local minimizers of the functional
    \[
    \iint_\Omega F(\rho(x,y)) \, d\mu(x) d\mu(y) \qquad \text{ s.t. } \mu\in\mathcal P(\Omega)
    \]
    in the $ d_\infty $ topology are discrete.
\end{corollary}
\noindent This is implied by the sectional curvature being continuous, and therefore bounded below on such manifold.

\section{Discussion}
\label{sec:discussion}
As mentioned in the previous section, the proof of Theorem~\ref{thm:main} applies without major changes to complete geodesic spaces with curvature bounded below. The angles between geodesics in this context are defined as Alexandrov angles.

The present result should be compared to the one of Bj\"orck \cite{Bjorck1956}, showing that for $ \delta > 2 $, the minimizers of 
\[
    \iint_\Omega -\|x-y\|^\delta \, d\mu(x) d\mu(y) \qquad \text{ s.t. } \mu\in\mathcal P(\Omega)
\]
on a compact $ \Omega\subset \mathbb R^d $ are supported on at most $ d+1 $ points. In a recent paper by Lim and McCann \cite{limIsodiametry2019,limGeometrical2020}, it is obtained that the global minimum of 
\[
    \iint_{\mathbb R^d} F_{\alpha,\beta}(\|x-y\|) \, d\mu(x) d\mu(y) \qquad \text{ s.t. } \mu\in\mathcal P(\mathbb R^d)
\]
is achieved on the equally weighted vertices of a regular simplex. Here $ F_{\alpha,\beta}(t) = \frac{t^\alpha}\alpha - \frac{t^\beta}\beta $, for $ \alpha \geq \beta  $ with $ \beta \geq 2 $ sufficiently large. In this case, the cardinality of the support of the global minimizer is therefore exactly $ d+1 $. Furthermore, when $ \Omega = \mathbb S^d $ with geodesic distance, it is known that tight spherical designs are the unique global minimizers for $ I_F $ with $ F(t) = |\cos(t)|^p $, for certain ranges of $ p $ \cite{bilykOptimal2019}. It would be interesting to find other cases where the cardinality of the supports of minimizers can be determined precisely.

It is useful to observe that the bound $ \alpha > 2 $ in Theorem~\ref{thm:main} is not sharp at least for some $ \Omega $, as demonstrated by Theorem~\ref{thm:bilyk}. 
\medskip

{\bf Acknowledgments. } A number of helpful conversations with Dmitriy Bilyk, Alexey Glazyrin, Ryan Matzke, and Josiah Park on the subject of this note are gratefully acknowledged.

\bibliographystyle{acm}
\bibliography{refs} 

\begin{thebibliography}{10}

\bibitem{bilykGeodesic2016}
{\sc Bilyk, D., and Dai, F.}
\newblock Geodesic distance {{Riesz}} energy on the sphere.
\newblock {\em ArXiv161208442 Math\/} (2016).

\bibitem{bilykStolarsky2016}
{\sc Bilyk, D., Dai, F., and Matzke, R.}
\newblock Stolarsky principle and energy optimization on the sphere.
\newblock {\em ArXiv161104420 Math\/} (2016).

\bibitem{bilykEnergy2019}
{\sc Bilyk, D., Glazyrin, A., Matzke, R., Park, J., and Vlasiuk, O.}
\newblock Energy on spheres and discreteness of minimizing measures.
\newblock {\em ArXiv190810354 Math-Ph\/} (2019).

\bibitem{bilykOptimal2019}
{\sc Bilyk, D., Glazyrin, A., Matzke, R., Park, J., and Vlasiuk, O.}
\newblock Optimal measures for p-frame energies on spheres.
\newblock {\em ArXiv190800885 Math\/} (2019).

\bibitem{Bjorck1956}
{\sc Bj{\"{o}}rck, G.}
\newblock {Distributions of positive mass, which maximize a certain generalized
  energy integral}.
\newblock {\em Ark. f{\"{o}}r Mat. 3}, 3 (feb 1956), 255--269.

\bibitem{bridsonMetric1999}
{\sc Bridson, M.~R., and Haefliger, A.}
\newblock {\em {Metric Spaces of Non-Positive Curvature}}.
\newblock {Springer Berlin Heidelberg}, Berlin, Heidelberg, 1999.
\newblock OCLC: 851367925.

\bibitem{carrilloGeometry2017}
{\sc Carrillo, J.~A., Figalli, A., and Patacchini, F.~S.}
\newblock Geometry of minimizers for the interaction energy with mildly
  repulsive potentials.
\newblock {\em Ann. Inst. Henri Poincare C Non Linear Anal. 34}, 5 (2017),
  1299--1308.

\bibitem{cheegerComparison2008}
{\sc Cheeger, J., and Ebin, D.~G.}
\newblock {\em Comparison Theorems in Riemannian Geometry}.
\newblock AMS Chelsea Publishing. {AMS Chelsea Publishing}, Providence, R.I,
  2008.
\newblock OCLC: ocn185095562.

\bibitem{jostRiemannian2011}
{\sc Jost, J.}
\newblock {\em Riemannian {{Geometry}} and {{Geometric Analysis}}}.
\newblock Universitext. {Springer Berlin Heidelberg}, Berlin, Heidelberg, 2011.

\bibitem{limIsodiametry2019}
{\sc Lim, T., and McCann, R.~J.}
\newblock Isodiametry, variance, and regular simplices from particle
  interactions.
\newblock {\em ArXiv190713593 Math-Ph Stat\/} (2019).

\bibitem{limGeometrical2020}
{\sc Lim, T., and McCann, R.~J.}
\newblock Geometrical bounds for the variance and recentered moments.
\newblock {\em ArXiv200111851 Math Stat\/} (2020).

\end{thebibliography}

\medskip
\noindent{\it Email address:\ }{\tt ovlasiuk@fsu.edu}
\end{document}